\documentclass[12pt]{amsart}
\usepackage{amsfonts,amssymb,amscd,amsmath,enumerate,url,verbatim}
\usepackage[all]{xy}

\theoremstyle{plain}

\newtheorem{theorem}{Theorem}[section]

\newtheorem{prop}[theorem]{Proposition}
\newtheorem{lemma}[theorem]{Lemma}

\newtheorem{obs}[theorem]{Observation}
\newtheorem{fact}[theorem]{Facts}
\theoremstyle{definition}

\newtheorem{definition}[theorem]{Definition}

\newcommand{\m}{\mathfrak{m}}

\begin{document}
\title[An observation on generalized Hilbert-Kunz functions]
{An observation on generalized Hilbert-Kunz functions}

\author{Adela Vraciu}
\address{Adela Vraciu\\Department of Mathematics\\University of South Carolina\\\linebreak 
Columbia\\ SC 29208\\ U.S.A.} \email{vraciu@math.sc.edu}


\thanks{Research partly supported by NSF grant  DMS-1200085}

\begin{abstract}
We prove that, under certain assumptions, generalized Hilbert-Kunz multiplicities can be expressed as linear combinations of classical Hilbert-Kunz multiplicities.
\end{abstract}
\maketitle

\section{Introduction}

Let $(R, \mathfrak{m})$ be a local ring of positive characteristic $p$ and Krull dimension $d$. Assume that $R$ is $F$-finite and has perfect residue field. The classical Hilbert-Kunz function and multiplicity are defined for $R$-modules of finite length, and in particular for $\m$-primary ideals.
\begin{definition}
Let $M$ be a finite length $R$-module. The Hilbert-Kunz function of $M$ is
$$
f_{HK}^M(n)= l(F^n(M))
$$
where $F^n(M)=M\otimes_R {^{n}R}$ denotes the $n$-fold iteration of the Frobenius functor, and the Hilbert-Kunz multiplicity of $M$ is
$$
e_{HK}(M)= \lim_{n \rightarrow \infty} \frac{f_{HK}^M(n)}{p^{nd}}.
$$
In particular, if  $I\subset R$ is an $\m$-primary ideal, the Hilbert-Kunz function of $R/I$ is 
$$f_{HK}^{R/I}(n):= l\left(\frac{R}{I^{[p^n]}}\right).$$
\end{definition}
These concepts were introduced by Kunz in \cite{Ku}. The fact that the limit in the definition of $e_{HK}(M)$ exists was proved by Monsky in \cite{Mo}.
Hilbert-Kunz functions and multiplicities have connections to the theory of tight closure and the classification of singularities, see \cite{Hu2}. 

Epstein and Yao in \cite{EY}  generalize the definition of  Hilbert-Kunz functions to the case $\mathrm{dim}(M)>0$ by using $0^{\mathrm th}$ local cohomology. The following is a special case of the definition in \cite{EY}, where a relative version is considered.
\begin{definition}
Let $M$ be a finitely generated $R$-module. The generalized Hilbert-Kunz function of $M$ is
$$
f_{gHK}^M(n)= l\left(H_{\m}^0(F^n(M))\right).
$$
\end{definition}
One drawback of this notion is that the limit $\displaystyle \lim_{n\rightarrow \infty} f_{gHK}^M(n)/p^{nd}$ that one would naturally want to use as the definition of the generalized Hilbert-Kunz multiplicity is not known to exist in general.  One is thus forced to define $e^+_{gHK}(M)$ and $e^-_{gHK}(M)$ as the limsup and the liminf respectively.

This existence of the limit is proved in \cite{DS} for modules that have finite projective dimension on the punctured spectrum, with certain assumptions on the ring.

A related problem that involves the $0^{th}$ local cohomology of Frobenius powers is the following long-standing conjecture, which will be referred to as the (LC) property in this note (the conjecture being that all rings of characteristic $p$ satisfy (LC)). We will need to use a more general version, where sequences of ideals other than Frobenius powers of a fixed ideal are involved.
\begin{definition}

({\bf 1.}) We say that a sequence of ideals $J_q$ indexed by $q=p^n$ satisfies the (LC) property if there exists $N$ such that $$\displaystyle \m^{Nq}H_{\m}^0\left(\frac{R}{J_q}\right)=0 \ \forall q=p^n.$$

({\bf 2.}) We say that the ring $R$ satisfies (LC) if for every ideal $J$ in $R$, the sequence of Frobenius powers  $J^{[q]}$ satisfies (LC).
\end{definition}


(LC) was previously studied in connection with the question of whether tight closure commutes with localization (which is now known to be false, see \cite{Br2}), and specifically whether a localization of a ring in which all ideals are tightly closed will continue to have this property (which is still an open question). See the discussion following Prop. 4.16 in \cite{HH1} for the connection between (LC) and the problem of localization of tight closure. See also \cite{Ab} for some work in this direction.

(LC) is known to hold in certain cases, such as  when $R$ is a graded ring and $J_q=J^{[q]}$ where $J$ is a homogeneous ideal with $\mathrm{dim}(R/J)=1$ (see \cite{Hu}, \cite{Vr}).

In \cite{DS} it is observed that if $R$ satisfies (LC) and countable prime avoidance, then $e^+_{gHK}(M)$ is finite for any finitely generated module $M$. The main result of this paper is a significant improvement of this observation.

More precisely, we show that if  $R$ satisfies (LC) and countable prime avoidance, then generalized Hilbert-Kunz functions of ideals can be expressed as linear combinations of classical Hilbert-Kunz functions (in particular, the generalized Hilbert-Kunz multiplicity exists, and can be expressed in terms of classical Hilbert-Kunz multiplicities of related ideals). Countable prime avoidance is a mild assumption, known to hold if the ring is complete, or if the residue field is uncountable, see \cite{Bu}.
By way of motivation for this work, we point out that Brenner's construction of an irrational Hilbert-Kunz multiplicity in \cite{Br} involved generalized Hilbert-Kunz multiplicities in an essential way. Namely, Brenner first constructs a (non $\m$-primary) ideal $I$ for which $e_{gHK}(R/I)$ is irrational, then he uses that to obtain  (in a non-explicit manner) a finite length module with irrational Hilbert-Kunz multiplicity, and as a final step he obtains an $\m$-primary ideal with irrational Hilbert-Kunz multiplicity.  If one could check that (LC) holds for the Frobenius powers of the ideal $I$ obtained in the first step of Brenner's construction, our result could then be used to obtain a more direct and explicit route to $\m$-primary ideals with irrational Hilbert-Kunz multiplicities.

We recall some basic facts and definitions.
\begin{definition}
Let $(R, \m)$ be a local ring, and let $I \subset R$ be an ideal. Then $\displaystyle I^{sat}= \bigcup_n (I : \m^n)$, and $\displaystyle H_{\m}^0(\frac{R}{I})=\frac{I^{sat}}{I}$.
\end{definition}
The following facts are well-known. We include a proof for the reader's convenience.
\begin{fact}\label{facts}
{\rm a.} Let $I=Q_1 \cap Q_2 \cdots \cap Q_t$ be a primary decomposition of $I$. Assume that $Q_t$ is the $\m$-primary component (possibly $Q_t=R$ if $\m$ is not an associated prime), and $Q_1, \ldots, Q_{t-1}$ are primary to non-maximal prime ideals. Then $I^{sat} = Q_1 \cap Q_2 \cap \cdots \cap Q_{t-1}$.

{\rm b.} If $s\in R$ does not belong to any associated prime ideal of $R/I$ except for the maximal ideal, then we have $(I:s) \subseteq I^{sat}$. If moreover $s\in \m^n$ and $\displaystyle \m^n H_{\m}^0(\frac{R}{I})=0$, then $(I:s)= I^{sat}$. 

{\rm c.} If $R$ satisfies countable prime avoidance, and $J_q$ is a sequence of ideals that satisfies (LC), there exists $s \in \m$ such that $$ H_{\m}^0(\frac{R}{J_q})= \frac{(J_q:s^q)}{J_q} \ \forall q=p^e$$
\end{fact}
\begin{proof}

Let $P_i$ denote the unique associated prime of $R/Q_i$ for $1 \le i \le t-1$. 

a. If $x \in I^{sat}$, then $m^n x \subseteq Q_i$ for some $n\ge 1$, and all $1 \le i \le t-1$.  Since $\m^n \not\subseteq P_i$, and $Q_i$ is $P_i$-primary, it follows that $x \in Q_i$. Conversely, if $x \in Q_1 \cap Q_2 \cap \ldots \cap Q_{t-1}$, it follows that $Q_t x \subseteq I$. Since $Q_t$ is $\m$-primary, we have $\m^n \subseteq Q_t$ for some $n$, and therefore $m^n x \subseteq I$, which proves $x \in I^{sat}$.

b. If $x \in (I:s)$, we have $sx \in Q_i$ for all $1 \le i \le t-1$. Since $s \notin P_i$, it follows that
$x \in Q_i$, and thus we have $(I:s) \subseteq  Q_1 \cap Q_2 \cap \ldots \cap Q_{t-1}= I^{sat}$. If $\displaystyle \m^n H_{\m}^0(\frac{R}{I})=0$, we have $\m^n I^{sat}\subseteq I$,  and therefore $s\in \m^n$ implies $s I^{sat}\subseteq I$, or equivalently $(I:s) \supseteq I^{sat}$.

c. By countable prime avoidance, we can pick $s_0$ outside of all associated primes of all $J_q$, with the exception of the maximal ideal. If $N$ is such that $\displaystyle \m^{Nq}H_{\m}^0(\frac{R}{J_q})=0$, then we let $s=s_0^N$ and observe that the conditions in (b) are satisfied when we use $s^q$ for $s$ and $J_q$ for $I$.
\end{proof}

\section{Main Result}

\begin{theorem}\label{main_theorem}
Assume that $R$ satisfies (LC) and countable prime avoidance.
Let $I$ be an ideal with $\mathrm{dim}(R/I)=c$. Then the generalized Hilbert-Kunz function of $R/I$, $f_{gHK}(R/I)$ can be expressed as a linear combination with integer coefficients that do not depend on $I$ of  Hilbert-Kunz functions of $\m$-primary ideals.
\end{theorem}

\begin{lemma}\label{isomorphism}
Assume that $(I:s)=(I:s^2)$ (note that this holds in particular when $I^{\mathrm{sat}}=(I:s)$). Then
$$
\frac{(I:s)}{I}\cong \frac{(I:s)+(s)}{I + (s)}
$$
\end{lemma}
\begin{proof}
Let $f$ denote the composition $\displaystyle \frac{(I:s)}{I} \hookrightarrow \frac{R}{I} \twoheadrightarrow \frac{R}{I+(s)}$ where the first map is inclusion and the second map is projection. Note that the assumption that $I:s=I:s^2$ implies that $f$ is injective. Indeed, if $x \in (I:s) \cap (I +(s))$, we can write $x=i+ ys$ with $i \in I$ and $y \in R$, and we have $ys \in (I:s)$, or $y \in (I:s^2) = (I:s)$, which implies $x \in I$.

Therefore $\displaystyle \frac{(I:s)}{I} \cong \mathrm{im}(f)= \frac{(I:s) + (s)}{I+(s)}$.
\end{proof}

\begin{lemma}\label{length}
Assume that $I^{\mathrm{sat}}=(I:s)$. Then
\begin{equation}\label{length_eq}
l\left( H_{\mathfrak{m}}^0\left(\frac{R}{I}\right)\right) = l\left(H_{\mathfrak{m}}^0\left(\frac{R}{I+(s)}\right)\right)-l\left(H_{\mathfrak{m}}^0\left(\frac{R}{(I:s)+(s)}\right)\right)
\end{equation}

\end{lemma}

\begin{proof}
Consider the short exact sequence 
$$
0 \rightarrow \frac{(I:s)}{I} \stackrel{f}{\longrightarrow} \frac{R}{I+(s)} \longrightarrow \frac{R}{(I:s)+(s)} \rightarrow 0
$$
given by the proof of  Lemma ~\ref{isomorphism}.

Since $\displaystyle \frac{(I:s)}{I}=H_{\mathfrak{m}}^0(\frac{R}{I})$, it is a zero-dimensional module, and therefore $\displaystyle H_{\mathfrak{m}}^1(\frac{(I:s)}{I})=0$. This ensure that exactness is preserved when we apply $H_{\mathfrak{m}}^0 ( {\underline{\   \ } } ) $ to the above short exact sequence, and the conclusion follows from the additivity of length on short exact sequences.
\end{proof}

Before embarking on the proof of the main theorem, we illustrate it for low dimensional cases, where we give explicit formulas for the linear combinations mentioned in the statement of the main theorem.

\begin{prop}\label{dim1}
Assume that $R$ satisfies (LC) and countable prime avoidance. Let $I \subset R$ be an ideal with $\mathrm{dim}(R/I)=1$. Then there exists $s \in \m$ such that $I+(s)$, $I+(s^2)$ are $\m$-primary, and
$$f_{gHK}^{R/I}=2 f_{HK}^{R/I+(s)}-f_{HK}^{R/I+( s^2)}.$$
\end{prop}
\begin{proof}
  Countable prime avoidance allows us to pick $s \in \m$ which is  not in any associated primes of any Frobenius power $I^{[q]}$ except for $\m$. According to (~\ref{facts}), we can pick such an $s$ which satisfies $\displaystyle H_{\m}^0(\frac{R}{I^{[q]}})=\frac{(I^{[q]}:s^q)}{I^{[q]}}$. Note that this choice of $s$ also implies that $I+(s)$ and $I+(s^2)$ are $\m$-primary. Lemma ~\ref{length} now gives 
\begin{equation}\label{lincomb}
l\left(H_{\m}^0\left( \frac{R}{I^{[q]}}\right)\right)= l\left(\frac{R}{I^{[q]}+(s^q)}\right)- l\left(\frac{R}{(I^{[q]}:s^q) +(s^q)}\right)
\end{equation} 
The first term on the right hand side of equation (\ref{lincomb}) is a $f_{HK}^{R/I+(s)}$. In order to deal with the second term, consider the map 
\begin{equation}\label{amap}
\frac{R}{(I^{[q]}:s^q)+(s^q)} \rightarrow \frac{R}{I^{[q]}+(s^{2q})}
\end{equation}
given by multiplication  by $s^q$.
It is easy to check that this map is injective, and the cokernel is $\displaystyle \frac{R}{I^{[q]}+(s^q)}$. The conclusion now follows  using additivity of length on short exact sequences together with (\ref{lincomb}).
\end{proof}

\begin{prop}\label{dimension2}
Assume $R$ satisfies (LC) and countable prime avoidance. Let $I\subset R$ be an ideal with $\mathrm{dim}(R/I)=2$.  Then there exist $s, t \in \m$ such that
$$
f_{gHK}^{R/I}= 2f_{gHK}^{R/I+( s)}-2f_{gHK}^{R/I+( s^2, st)} + f_{gHK}^{R/I+( s^2, st^2)}
$$
where the terms on the right hand side are generalized Hilbert-Kunz functions of one-dimensional ideals.

\end{prop}
\begin{proof}
Pick $s \in \m$ as in the proof of Proposition ~\ref{dim1}, such that $\displaystyle H_{\m}^0(R/I^{[q]})=\frac{ (I^{[q]}:s^q)}{I^{[q]}}$. Apply  Lemma~\ref{length}  to $I^{[q]}$. Note that the first term in equation (\ref{length_eq}) is $f_{gHK}^{R/I+(s)}$,  a generalized Hilbert-Kunz function of a one-dimensional ideal. Now we  calculate the second term. 
Observe that the family of ideals $(I^{[q]}:s^q)+(s^q)$ satisfy (LC), since applying $H_{\m}^0( \underline{\ \ })$ to the map (\ref{amap}) shows that  $H_{\m}^0(R/(I^{[q]}:s^q)+(s^q))$ is a submodule of $H_{\m}^0(R/I^{[q]}+(s^{2q}))$, and the latter local cohomology module is annihilated by $m^{Nq}$ for some $N$ that does not depend on $q$ from the assumption that the ring satisfies (LC). Thus, we can pick $t$ such that 
$$
H_{\m}^0\left(\frac{R}{(I^{[q]}:s^q)+(s^q)}\right) = \frac{(I^{[q]}:s^q+(s^q)):t^q}{(I^{[q]}:s^q) +(s^q)},
$$
and Lemma~\ref{length} gives
\begin{equation}\label{second_term}
l\left(H_{\m}^0\left(\frac{R}{(I^{[q]}:s^q)+(s^q)}\right)\right)= l\left(\frac{R}{(I^{[q]}:s^q)+(s^q, t^q)}\right)- l\left(\frac{R}{(I^{[q]}:s^q+(s^q)):t^q +( t^q)}\right)
\end{equation}
(note that the two quotients on the right hand side are zero dimensional).
In order to express the two terms on the right hand side of equation (\ref{second_term}) in terms of Hilbert-Kunz multiplicities of zero-dimensional ideals, we consider the following two short exact sequences, where the leftmost maps are multiplication by $s^q$ and $t^q$ respectively:
\begin{equation}\label{ses1}
0 \rightarrow \frac{R}{(I^{[q]}:s^q)+( s^q, t^q)} \rightarrow \frac{R}{I^{[q]}+( s^{2q}, s^qt^q)}\rightarrow \frac{R}{I^{[q]}+(s^q)}\rightarrow 0
\end{equation}
\begin{equation}\label{ses2}
0 \rightarrow \frac{R}{(I^{[q]}:s^q, s^q):t^q+( t^q)} \rightarrow \frac{R}{(I^{[q]}:s^q)+(s^q, t^{2q})}\rightarrow \frac{R}{(I^{[q]}:s^q)+( s^q, t^q)}\rightarrow 0
\end{equation}

Since the leftmost terms in (\ref{ses1}) and (\ref{ses2}) are zero-dimensional, exactness is preserved when applying $H_{\m}^0(\underline{\ \ })$. 
The two rightmost terms of the short exact sequence (\ref{ses1}) become $f_{gHK}^{R/I+(s^2, st)}$ and $f_{gHK}^{R/I+(s)}$ respectively, which are generalized Hilbert-Kunz functions of one-dimensional ideals.

The two rightmost terms of the short exact sequence (\ref{ses2}) are zero dimensional, and in order to evaluate their lengths we use the short exact sequences
\begin{equation}\label{ses3}
0 \rightarrow \frac{R}{(I^{[q]}:s^q)+(s^q, t^{2q})} \rightarrow \frac{R}{I^{[q]}+(s^{2q}, s^q t^{2q})}\rightarrow \frac{R}{I^{[q]}+(s^q)}\rightarrow 0,
\end{equation}
and
\begin{equation}\label{ses4}
0\rightarrow \frac{R}{(I^{[q]}:s^q)+(s^q, t^q)} \rightarrow \frac{R}{I^{[q]}+(s^{2q}, s^qt^q)}\rightarrow \frac{R}{I^{[q]}+( s^q)}\rightarrow 0
\end{equation}
Exactness is preserved in (\ref{ses3}) and (\ref{ses4}) when applying $H_{\m}^0( \underline{\ \ })$, and now the two right-most terms become $f_{gHK}^{R/I+(s^2, st^2)}$ and $f_{gHK}^{R/I+(s)}$ in (\ref{ses3}), and $f_{gHK}^{R/I+(s^2, st)}$ and $f_{gHK}^{R/I+(s)}$ in (\ref{ses4}). The conclusion follows from the additivity of length on short exact sequences.

\end{proof}
The proof of the main theorem will use the following notation:
\begin{definition}\label{notation}
Let $J, K\subset R$ and $x\in R$. We define $J*_0 xK:= J+xK$, and $J*_1 x:= (J: x) + (x)$. We will frequently write $J:_{\epsilon} xK$ with $\epsilon\in \{0, 1\}$. It will be understood that if $\epsilon=1$ then $K=R$ and $J*_1 xR$ just means $J*_1x$. When we write $J*_{\epsilon_1}x_1K_1*_{\epsilon_2}\cdots *_{\epsilon_s}x_sK_s$ it will be understood that the order of operation is from left to right.

We will make the assumption that $K$ and $(x)$ are not contained in any associated prime ideal of $J$ except for $\m$, and $\mathrm{dim}(R/J)>0$. Then we have $\mathrm{dim}(R/J*_{\epsilon} xK)=\mathrm{dim}(R/J)-1$. When we write $J*_{\epsilon_1}x_1K_1*_{\epsilon_2}\cdots *_{\epsilon_s}x_sK_s$, we will make the assumption that for all $i \le s$, $x_i$ and $K_i$ are not contained in any associated prime ideal of $J*_{\epsilon_1}x_1K_1*_{\epsilon_2}\cdots *_{\epsilon_{i-1}}x_{i-1}K_{i-1}$.
\end{definition}

\begin{lemma}\label{recursive}

For all $c>0$ and all $(\epsilon_1, \ldots, \epsilon_c) \in \{0, 1\}^c$, there exist integers $d_{(\epsilon_1, \ldots, \epsilon_c)}$ such that if $J\subset R$ is an ideal with $\mathrm{dim}(R/J)=c$, the following hold:

{\rm a.} We can pick elements $x_1, \ldots, x_c$ such that 
$$
l\left(H_{\m}^0\left(\frac{R}{J}\right)\right) = \sum d_{\epsilon_1, \ldots, \epsilon_c} l\left(\frac{R}{J*_{\epsilon_1}x_1*_{\epsilon_2}  \ldots *_{\epsilon_c}x_c}\right)
$$
where the summation is over all $(\epsilon_1, \ldots, \epsilon_x) \in \{0, 1\}^c$ and all the $J*_{\epsilon_1}x_1*_{\epsilon_2} \ldots *_{\epsilon_c} x_c$ are $\m$-primary ideals.

Assume moreover that $R$ satisfies countable prime avoidance and has the (LC) property. Then:

{\rm b.} We can pick $y_1, \ldots, y_c$ independent of $q$, such that 
\begin{equation}\label{linearcomb}
l\left(H_{\m}^0\left(\frac{R}{J^{[q]}}\right)\right)=\sum d_{\epsilon_1, \ldots, \epsilon_c} l\left(\frac{R}{J^{[q]}*_{\epsilon_1}y_1^q*_{\epsilon_2}  \ldots *_{\epsilon_c}y_c^q}\right) \ \forall q=p^e
\end{equation}
where the summation is over all $(\epsilon_1, \ldots, \epsilon_c) \in \{0, 1\}^c$ and all the $J^{[q]}*_{\epsilon_1}y_1^q*_{\epsilon_2} \ldots *_{\epsilon_c} y_c^q$ are $\m$-primary ideals.

{\rm c. } For every $s<c$,  $\epsilon'_1, \ldots, \epsilon'_s\in \{0, 1\}$,  $y_1, \ldots, y_s \in R$ and $K_1, \ldots, K_s$ ideals in $R$ satisfying the assumption in (\ref{notation}), we can pick $z_{s+1}, \ldots, z_c$ such that
$$
l\left(H_{\m}^0\left(\frac{R}{J^{[q]}*_{\epsilon_1'}*y_1^qK_1^{[q]}*_{\epsilon_2'} \ldots *_{\epsilon_s'}y_s^qK_s^{[q]}}\right)\right) =
$$
$$
 \sum d_{\epsilon_{s+1}, \ldots, \epsilon_c} l\left(\frac{R}{J*_{\epsilon'_1}y_1^qK_1^{[q]}*_{\epsilon'_2}  \ldots *_{\epsilon'_s}y_s^qK_s^{[q]}*_{\epsilon_{s+1}}z_{s+1}^q *_{\epsilon_{s+2}} \ldots *_{\epsilon_c} z_c^q}\right)
$$
where the summation is over all the choices of $(\epsilon_{s+1}, \ldots, \epsilon_c) \in \{0, 1\}^{c-s}$ and all the ideals appearing in the quotients on the right hand side are $\m$-primary.
 

\end{lemma}
\begin{proof}

(a). We pick $x_i$ recursively to be outside of all associated primes of $J*_{\epsilon_1}x_1 *_{\epsilon_2}*\cdots *_{\epsilon_{i-1}}x_{i-1}$ except for the maximal ideal, for all $\epsilon_1, \ldots, \epsilon_{i-1} \in \{0, 1\}$. Note that the ideals
$J':=J*_{\epsilon_1}x_1 *_{\epsilon_2}*\cdots*_{\epsilon_{i-1}}x_{i-1}$ constructed this way have dimension $c-i+1$, and we may choose $x_i$ such that
 $\displaystyle H_{\m}^0\left(\frac{R}{J'}\right)= \frac{J':x_i}{J'}$.  Lemma~\ref{length} shows that the length of this local cohomology module can be computed as 
$$
\left(H_{\m}^0\left(\frac{R}{J'*_0x_i}\right)\right) - l\left(H_{\m}^0\left(\frac{R}{J'*_1x_i}\right)\right)
$$
The claim now follows by induction on $c$.

(b). and (c). Induct on $c$. Note that (LC) for the ideal $J$  gives that we can pick $x_1 = y_1^q$. Then the ideal $J^{[q]}*_0x_1= (J, y_1)^{[q]}$ is a Frobenius power of an ideal of dimension $c-1$, and by the inductive hypothesis we can pick $x_i = y_i^q$  ($i=2, \ldots, c$)  that satisfy the desired conclusion for $J':=J^{[q]}*_0x_1$.
 As in the proof of Proposition (\ref{dimension2}), we note that the family of ideals  $J_q'':=J^{[q]}*_1y_1^{q_1}$ satisfy (LC), since $H_{\m}^0(R/J_q'')$ is a submodule of $H_{\m}^0(R/J^{[q]}+(y_1^{2q}))$. Therefore we can pick $x_2=y_2^q$ such that
$$
l\left(H_{\m}^0\left(\frac{R}{J_q''}\right)\right) = l\left(H_{\m}^0\left(\frac{R}{J''_q*_0x_2^q}\right)\right) - l\left(H_{\m}^0\left( \frac{R}{J''_q*_1x_2^q}\right)\right)
$$
(via Lemma (\ref{length})).

We claim that we can pick $y_i$ recursively such that for all $i<c$ and all $\epsilon_1, \ldots, \epsilon_i \in \{0, 1\}$ we have
$$
H_{\m}^0\left(\frac{R}{J_{q, (\epsilon_1, \ldots, \epsilon_i)}}\right)= \frac{J_{q, (\epsilon_1, \ldots, \epsilon_i)}:y_{i+1}^q}{J_{q, (\epsilon_1, \ldots, \epsilon_i)}},
$$
where $J_{q, (\epsilon_1, \ldots, \epsilon_i)}:= J^{[q]}*_{\epsilon_1}y_1^q *_{\epsilon_2} \cdots *_{\epsilon_i}y_i^q$  

 Assume that all $y_j$ for $j\le i$ are picked. In order to be able to pick $y_{i+1}$, we need to show that the family of ideals $J_{q, (\epsilon_1, \ldots, \epsilon_i)}$  satisfy (LC) for every $\epsilon_1, \ldots, \epsilon_i \in \{0, 1\}$.
We claim more generally that for any elements $f_1, \ldots, f_N$, ideals $K_1, \ldots, K_N$, and $\epsilon_1, \ldots, \epsilon_N \in \{0, 1\}$, the family of  ideals $J_{q, (\epsilon_1, \ldots, \epsilon_N)}:=J^{[q]}*_{\epsilon_1}f_1^qK_1^{[q]}*_{\epsilon_2}f_2^qK_2^{[q]}*_{\epsilon_3} \ldots *_{\epsilon_N}f_N^qK_N^{{q}}$ satisfy (LC). This will also justify the statement in (c).

Let $j$  be the largest index with $\epsilon_j=1$ (we consider $j$ to be zero if all $\epsilon_j=0$). We prove the claim by induction on $j$. The case $j=0$ is clear 
from the assumption that $R$ satisfies (LC), since the ideals under consideration are Frobenius powers of a fixed ideal. Let $J'_q:= J^{[q]}*_{\epsilon_1} f_1^qK_1^{[q]}*_{\epsilon_2} \ldots *_{\epsilon_{j-1}}f_{j-1}^qK_{j-1}^{[q]}$, so that  $$J_{q, (\epsilon_1, \ldots, \epsilon_i)}= J'_q :f_j^q + (f_j^q)+ f_{j+1}^qK_{j+1}^{[q]}+\ldots +  f_N^qK_N^{[q]}.$$

 Multiplication by $f_j^q$ gives an injection
$$
\frac{R}{J_{q, (\epsilon_1, \ldots, \epsilon_i)}   } \hookrightarrow \frac{R}{J'_q +(f_j^{2q})+ f_j^qf_{j+1}^qK_{j+1}^{[q]}+\ldots + f_j^qf_N^qK_N^{[q]}}. 
$$
The denominator of the second quotient can be described as $$J^{[q]}*_{\epsilon'_1}f_1'^q K_1^{[q]}*_{\epsilon_2'}f_2'^qK_2^{[q]}* \ldots * {\epsilon'_N}f_N'^qK_N^{[q]},$$ where $\epsilon_i'=\epsilon_i$ for all $i\ne j$, $\epsilon_j'=0$,  $f_i'=f_i$ for $i<j$, and $f_i'=f_jf_i$ for $i \ge j$. The inductive hypothesis tells us that this family of ideals satisfy (LC), and therefore, since $H_{\mathfrak{m}}^0(\underline{ \ })$ is left exact,  the family of ideals $J_{q, (\epsilon_1, \ldots, \epsilon_N)}$ do as well.

\end{proof}

Now we are ready to give the proof of the main theorem.
\begin{proof}
Let $\mathrm{dim}(R/J)=c$, and let $y_1, \ldots, y_c$ be such that equation (\ref{linearcomb}) from Lemma (\ref{recursive}) holds. We claim that for all $(\epsilon_1, \ldots, \epsilon_c) \in \{0, 1\}^c$, the length 
$$
l\left(\frac{R}{J^{[q]}*_{\epsilon_1}y^q_1*_{\epsilon_2} y^q_2*_{\epsilon_3} \ldots *_{\epsilon_c}y^q_c}\right)
$$
can be expressed as a linear combination (with integer coefficients that do not depended on $J$ and $q$) of generalized Hilbert-Kunz functions of ideals of  dimension $\le c-1$.
Due to Lemma (\ref{recursive}), this will show that $f_{gHK}^{R/J}$ is a linear combination of generalized Hilbert-Kunz functions of ideals of dimension $\le c-1$. Repeated applications of this fact will  then allow us to reduce to Hilbert-Kunz functions of zero-dimensional ideals, as desired.

More generally, we claim that if $y_1, \ldots, y_c$, $K_1, \ldots, K_c$ are such that the assumption in Definition (\ref{notation}) is satisfied, then the length of 
$$
l\left(\frac{R}{J^{[q]}*_{\epsilon_1}y_1^qK_1^{[q]}*_{\epsilon_2} y^q_2K_2^{[q]} *_{\epsilon_3} \ldots *_{\epsilon_c}y^q_cK_c^{[q]}}\right)
$$
can be expressed as a linear combination (with integer coefficients independent of $q$) of generalized Hilbert-Kunz functions of ideals of  dimension $\le c-1$.

 We induct on $N(\epsilon_1, \ldots, \epsilon_c):=\sum_{i=1}^c \epsilon_i 2^{c-i}$. The base case is $\epsilon_1 = \ldots = \epsilon_c =0$, in which case the function under consideration is itself a Hilbert-Kunz function of a zero-dimensional ideal.
For the inductive step, let $j$ be the largest index for which $\epsilon_j=1$. Let  $J'_q= J^{[q]}*_{\epsilon_1}y_1^qK_1^{[q]} *_{\epsilon_2} \ldots *_{\epsilon_{j-1}}y_{j-1}^qK_{j-1}^{[q]}$. Note that
$\mathcal{J}:= J^{[q]}*_{\epsilon_1}y_1^qK_1^{[q]}*_{\epsilon_2} y^q_2K_2^{[q]}*_{\epsilon_3} \ldots *_{\epsilon_c}y^q_cK_c^{[q]}$ can be written as $J'_q:y_j^q + y_j^q+ y_{j+1}^qK_{j+1}^{[q]} + \ldots + y_c^q K_c^{[q]}$.
Multiplication by $y_j^q$ gives an injection 
\begin{equation}\label{injection}
\frac{R}{\mathcal{J}} \hookrightarrow \frac{R}{J'_q+y_j^{2q}+y_j^q(y_{j+1}^qK_{j+1}^{[q]}+ \cdots + y_c^q K_c^{[q]})}
\end{equation}
The denominator of the second quotient above can be written as $J'_q*_0y_j^qK'^{[q]}$ where $K'=(y_j^q)+y_{j+1}^qK_{j+1}^{[q]}+ \ldots + y_c^qK_c^{[q]}$. The cokernel of the map  (\ref{injection}) is 
$\displaystyle \frac{R}{J'_q+y_j^q}$. Since $\displaystyle \frac{R}{\mathcal{J}}$ is zero-dimensional, applying $H_{\m}^0$ to the resulting short exact sequence preserves exactness. Therefore, the length of $\displaystyle H_{\m }^0(\frac{R}{\mathcal{J}})$ is
\begin{equation}\label{conclusion}
l\left(H_{\m}^0\left( \frac{R}{J'_q *_0 y_j^q K'^{[q]}}\right)\right) - l\left(H_{\m}^0\left(\frac{R}{J'_q*_0y_j^q}\right)\right)
\end{equation}
Now apply the last part of  Lemma ~\ref{recursive} to find $z_{j+1}, \ldots, z_c$ such that the two terms in (\ref{conclusion}) are linear combinations of 
$\displaystyle l \left(\frac{R}{J'_q*_0y_j^qK'^{[q]}*_{\epsilon'_{j+1}}z_{j+1}^c *_{\epsilon'_{j+2}}\ldots *_{\epsilon'_c}z_c^q}\right)$, and $\displaystyle l  \left(\frac{R}{J'_q*_0y_j^q*_{\epsilon'_{j+1}}z_{j+1}^q *_{\epsilon'_{j+2}}\ldots *_{\epsilon'_c}z_c^q}\right)$ respectively, where $(\epsilon'_{j+1}, \ldots, \epsilon'_c)$ range through $\{0, 1\}^{c-j}$. We can think of the ideals in the denominators as $J^{[q]}*_{\epsilon_1'} z_1^q{K'}_1^{[q]}*_{\epsilon_2'} \cdots *_{\epsilon_c'}z_c^q{K'}_c^{[q]}$ where for $i<j$ we have $\epsilon_i'=\epsilon_i$, $z_i=y_i$, $K_i'=K_i$, for $i>j$ we have $K_i'=R$,  $\epsilon_i=0$, and for $i=j$ we have $\epsilon_j'=0$, $z_j=y_j$, $K_j'=K'$ (in the case of the first quotient), or $K_j'=R$ (in the case of the second quotient). We have 
$$N(\epsilon_1', \ldots, \epsilon_c')=\sum_{i=1}^{j-1}\epsilon_i2^{c-i}+ \sum_{i=j+1}^c \epsilon'_i2^{c-i}\le
$$
$$ \sum_{i=1}^{j-1}\epsilon_i2^{c-i}+\sum_{i=j+1}^c2^{c-i}= \sum_{i=1}^{j-1}\epsilon_i2^{c-i}+2^{c-j}-1= N(\epsilon_1, \ldots, \epsilon_c)-1$$
and the desired conclusion follows from the inductive hypothesis.

\end{proof}

\begin{obs}
The proof of Theorem (\ref{main_theorem}) only uses the (LC) assumption for Frobenius powers of ideals of the form $J+y_1K_1 + \ldots y_jK_j$, where $y_1, \ldots, y_c$ are such that equation (\ref{linearcomb}) in Lemma (\ref{recursive}) holds, and $K_1, \ldots, K_c$ are ideals generated by certain products of $y_1, \ldots, y_c$. In particular, we only need to assume (LC) for Frobenius powers of ideals containing $J$.

Furthermore,  if $R$ is graded and $J$ is a homogeneous ideal, then we can pick $y_1, \ldots, y_c$ to be homogeneous elements, and thus it suffices to assume (LC) for Frobenius powers of homogeneous ideals containing $J$.
\end{obs}

\end{document}